\documentclass[11pt]{article}
\usepackage[margin=1in]{geometry}

\usepackage{enumerate}

\usepackage{amsmath,amssymb,amsthm}
\usepackage{mathrsfs,amsbsy}

\usepackage{kbordermatrix}

\usepackage{indentfirst}
\usepackage{algorithm,algorithmic}

\usepackage{graphicx}
\usepackage{float}
\usepackage{subcaption}

\usepackage{multirow}

\usepackage[normalem]{ulem}

\usepackage{cprotect}
\usepackage{xcolor}

\usepackage{pgfplots}
\pgfplotsset{compat=newest}

\usepackage{url}
\usepackage{appendix}

\def\T{\mathrm{T}}
\def\F{{\rm F}}
\def\R{{\mathbb R}}

\newcommand{\sig}{\sigma}
\newcommand{\gap}{\gamma}

\newcommand{\lam}{\lambda}
\newcommand{\Lam}{\Lambda}

\newcommand{\fl}[1]{\widehat{#1}}
\newcommand{\mc}[1]{\mathcal{#1}}

\newcommand{\PARAGRAPH}[1]{\medskip}

\DeclareMathOperator{\diag}{diag}

\newcommand{\twobyone}[2]{
        \left[ \begin{array}{c}
        #1  \\  #2
        \end{array} \right] }

\newtheorem{theorem}{\textbf{Theorem}}[section]
\newtheorem{lemma}{\textbf{Lemma}}[section]
\newtheorem{remark}{\textbf{Remark}}[section] 

\newtheorem{example}{\textbf{Example}}[section]

\numberwithin{equation}{section}

\title{Backward Stability of Explicit External Deflation 
for the Symmetric Eigenvalue Problem}
\author{Chao-Ping Lin\thanks{Department of Mathematics, 
University of California, Davis, CA 95616, USA, {\tt cplin@ucdavis.edu}} 
\and Ding Lu\thanks{Department of Mathematics, 
University of Kentucky, Lexington, KY 40506, USA, {\tt Ding.Lu@uky.edu}} 
\and Zhaojun Bai\thanks{Department of Mathematics 
and Department of Computer Science, University of California, 
Davis, CA 95616, USA, {\tt zbai@ucdavis.edu}} 
}  
\date{\today} 

\begin{document}

\maketitle

\begin{abstract}
A thorough backward stability analysis of 
Hotelling's deflation, an explicit external deflation procedure 
through low-rank updates for computing many eigenpairs of 
a symmetric matrix, is presented. 
{Computable upper bounds}
of the loss of the orthogonality of the computed eigenvectors and 
the symmetric backward error norm of the computed eigenpairs are 
derived. Sufficient conditions for the backward stability of 
the explicit external deflation procedure are revealed. 
Based on these theoretical results, 
the strategy for achieving numerical backward stability 
by dynamically selecting the shifts is proposed. 
Numerical results are presented to corroborate the theoretical 
analysis and to demonstrate the stability of the 
procedure for computing many eigenpairs
of large symmetric matrices arising from applications. 
\end{abstract}


\section{Introduction}

\PARAGRAPH{Problem.} 
Let $A$ be an $n\times n$ real symmetric matrix with ordered
eigenvalues
$\lambda_1 \leq \lambda_2 \leq \cdots \leq \lambda_n$, and
the corresponding eigenvectors $v_1, v_2, \ldots, v_n$.
Let $\mc{I}=[\lambda_{\rm low}, \lambda_{\rm upper}]$ 
be an interval containing 
the eigenvalues $\lambda_1, \lambda_2, \ldots, \lambda_{n_e}$
at the lower end of the spectrum of $A$.
We are interested in computing the partial eigenvalue decomposition
\begin{equation} \label{eq:pe}  
A V_{n_e} = V_{n_e} \Lambda_{n_e},
\end{equation} 
where $\Lambda_{n_e} 
= \mbox{diag}(\lambda_1, \lambda_2, \ldots, \lambda_{n_e})$, 
$V_{n_e} = [v_1,v_2,\ldots,v_{n_e}]$ 
and $V^{\T}_{n_e} V_{n_e} = I_{n_e}$. 
We are particularly interested in the case where 
the interval is of the width 
$|\mathcal{I}| = \lambda_{\rm upper} - \lambda_{\rm low}
\leq \frac{1}{2}\|A\|_2$,
and may contain a large number of eigenvalues of $A$. 

A majority of subspace projection methods for 
computing the partial eigenvalue decomposition~\eqref{eq:pe},   
such as the Lanczos methods \cite{Parlett1998,Bai2000}, 
typically converge first to the eigenvalues on the periphery of the spectrum.
A projection subspace of a large dimension
is necessary to compute the eigenvalues located deep inside the spectrum.
Maintaining the orthogonality of a large projection subspace is
computationally expensive.
To reduce the dimensions of projection subspaces and accelerate 
the convergence, a number of techniques such as 
the restarting \cite{Sorensen1992,Lehoucq1996,Wu2000} 
and the filtering \cite{LiXi2016} have been developed.


This paper revisits a classical deflation technique 
known as Hotelling's 
deflation \cite{Hotelling1933,Wilkinson1965,Parlett1998}.  
Hotelling's deflation computes the 
partial eigenvalue decomposition~\eqref{eq:pe} through 
a sequence of low-rank updates.
In the simplest case, supposing that the lowest eigenpair 
$(\lam_1,v_1)$ of $A$ is computed by an eigensolver,
by choosing a real shift $\sigma_1$ and defining the rank-one update
matrix 
\begin{align*}
A_1 \equiv A + \sigma_1 v_1v_1^\T,
\end{align*}
the eigenpair $(\lam_1,v_1)$ of $A$ 
is displaced to an eigenpair $(\lam_1+\sigma_1,v_1)$ of $A_1$,
while all the other eigenpairs of $A$ and $A_1$ are the same. 
If the shift $\sigma_1 > \lam_{\rm upper} - \lam_1$,
then the eigenpair $(\lam_1,v_1)$ is shifted outside the interval $\mc{I}$
and the eigenpair $(\lam_2,v_2)$ of $A$ 
becomes the lowest eigenpair of $A_1$. Subsequently, we
can use the eigensolver again to compute the 
lowest eigenpair $(\lam_2,v_2)$ of $A_1$. 
The procedure is repeated until all the eigenvalues
in the interval $\mc{I}$ are found.
Since the low-rank update of the original $A$ can be done outside of 
the eigensolver, we refer to this strategy as 
an explicit external deflation (EED) 
procedure, to distinguish from the deflation techniques used
inside an eigensolver, such as TRLan \cite{Wu2000} 
and ARPACK \cite{Lehoucq1998}.  
Recently the EED has been combined with Krylov subspace methods to 
progressively compute the eigenpairs toward the interior of 
the spectrum of symmetric eigenvalue problems 
\cite{Money2005,Yamazaki2019}, and extended to structured
eigenvalue problems~\cite{Bai2016,Chen2018}.

In the presence of finite precision arithmetic, 
the EED procedure 
is susceptible to numerical instability
due to the accumulation of numerical errors from previous 
computations; see, e.g., \cite[p.~585]{Wilkinson1965} and \cite[p.~96]{Saad2011}. 
In \cite[Chap.~5.1]{Parlett1998}, 
Parlett showed that the change to the smallest eigenvalue 
in magnitude by deflating out the largest computed eigenvalue 
would be at the same order of the round-off error incurred. 
In \cite{Saad1989}, for nonsymmetric matrices, Saad derived 
an upper bound on the backward error norm and
showed that the stability is determined by 
the angle between the computed eigenvectors and the deflated subspace.
Saad concluded that ``if the angle between the computed eigenvector
and the previous invariant subspace is small at every step, 
the process may quickly become unstable. On the other hand, 
if this is not the case then the process is quite safe, 
for small number of requested eigenvalues."

However, none of the existing studies had examined the 
effect of the choice of shifts on the numerical stability.   
In this paper, we present a thorough backward stability analysis 
of the EED procedure. We first derive governing equations 
of the EED procedure in the presence of finite precision arithmetic, 
and then derive upper bounds of the loss of the orthogonality 
of the computed eigenvectors and the symmetric backward error norm of 
the computed eigenpairs. From these upper bounds, 
under mild assumptions, we conclude that 
the EED procedure is backward stable if the shifts $\sigma_j$ are
dynamically chosen such that 
(i) the spectral gaps,
defined as the separation between 
the computed eigenvalues and shifted ones, 
are maintained at the same order of the norm of the matrix $A$, and 
(ii) the shift-gap ratios, defined as the ratios between
the largest shift in magnitude and the spectral gaps, 
are at the order of one. 
Our analysis implies that the re-orthogonalization
of computed eigenvectors and
the large angles between the computed eigenvectors
and the previous invariant subspaces are unnecessary.
Based on theoretical analysis, 
we provide a practical shift selection scheme 
to guarantee the backward stability.
For a set of test matrices from applications, we demonstrate that the 
EED in combination with an established eigensolver can 
compute a large number of eigenvalues with backward stability.

The rest of the paper is organized as follows. 
In Section~\ref{sec:EED}, we present the governing equations  
of the EED in finite precision arithmetic.
In Section~\ref{sec:backErr}, we derive the upper bounds
on the loss of orthogonality of computed eigenvectors and
the symmetric backward error norm of computed eigenpairs.   
From these upper bounds, we derive conditions for the backward stability
of the EED procedure. 
In Section~\ref{sec:EEDpractice}, we discuss how to 
properly choose the shifts to satisfy the stability conditions
and outline an algorithm for combining the EED procedure with 
an eigensolver.  Numerical results to verify the
sharpness of the upper bounds and to demonstrate the backward 
stability of the EED for symmetric eigenvalue problems from
applications are presented in Section~\ref{sec:examples}.
Concluding remarks are given in section~\ref{sec:conclude}.

Following the convention of matrix computations,
we use the upper case letters for matrices and the lower case letters for vectors.
In particular, we use $I_n$ for the identity matrix of dimension $n$.
If not specified, the dimensions of matrices and vectors 
conform to the dimensions used in the context. 
We use $\cdot^\T$ for transpose, and $\|\cdot\|_2$ and $\|\cdot\|_\F$ 
for $2$-norm and the Frobenius norm, respectively.
The range of a matrix $A$ is denoted by $\mc{R}(A)$.
We also use $\sigma_{\min}(\cdot)$ to denote the minimal singular values of a matrix.
Other notations will be explained as used.

%
%

\section{EED in finite precision arithmetic} \label{sec:EED}

Suppose that we use an established eigensolver called EIGSOL,  
such as TRLan \cite{Wu2000} and ARPACK \cite{Lehoucq1998}, 
to successfully compute the lowest eigenpair $(\fl{\lam}, \fl{v})$ of $A$ with 
\begin{equation*}
A\fl{v} = \fl{\lam}\fl{v} + \eta, 
\end{equation*} 
where $\|\fl{v}\|_2 = 1$ and the residual vector $\eta$ satisfies 
\begin{equation} \label{eq:toldef} 
\|\eta\|_2 \leq tol\cdot \|A\|_2
\end{equation} 
for a prescribed convergence tolerance $tol$.
The EED procedure starts with computing the lowest eigenpair 
$(\fl{\lam}_1,\fl{v}_1)$ of $A$ by EIGSOL satisfying 
\begin{align*}
A\fl{v}_1 = \fl{\lam}_1\fl{v}_1 + \eta_1,
\end{align*}
where $\fl{\lam}_1 \in \mc{I}=[\lambda_{\rm low}, \lambda_{\rm upper}]$,
and the residual vector $\eta_1$ satisfies \eqref{eq:toldef}.
At the first EED step, we choose a shift $\sigma_1$ and define 
\begin{align*}
\fl{A}_1\equiv A + \sigma_1\fl{v}_1\fl{v}_1^\T.
\end{align*}
By choosing the shift $\sigma_1 > \lambda_{\rm upper} - \fl{\lam}_1$,  
the lowest eigenpair of $\fl{A}_1$ is an approximation of 
the second eigenpair $({\lam}_2, {v}_2)$ of $A$.
Subsequently, we use EIGSOL to compute the lowest eigenpair 
$(\fl{\lam}_2,\fl{v}_2)$ of $\fl{A}_1$ satisfying 
\begin{align*}
\fl{A}_1 \fl{v}_2 = \fl{\lam}_2\fl{v}_2 + \eta_2,
\end{align*}
where
the residual vector $\eta_2$ satisfies \eqref{eq:toldef}.
Meanwhile, expressing the computed eigenpair 
$(\fl{\lam}_1,\fl{v}_1)$ in terms of $\fl{A}_1$, we have
\[
\fl{A}_1 \fl{v}_1 =  (\fl{\lam}_1 + \sigma_1)\fl{v}_1 + \eta_1.
\] 
Proceeding to the second EED step, we choose a shift $\sigma_2$ and define 
\[
\fl{A}_2\equiv \fl{A}_1 + \sigma_2\fl{v}_2\fl{v}_2^\T 
	= A + \fl{V}_2\Sigma_2\fl{V}_2^\T,
\]
where 
$\fl{V}_2 = [\fl{v}_1,\fl{v}_2]$ and $\Sigma_2 = \diag(\sigma_1,\sigma_2)$.
By choosing the shift $\sigma_2 > \lambda_{\rm upper} - \fl{\lam}_2$,  
the lowest eigenpair of $\fl{A}_2$ is 
an approximation of the third eigenpair $({\lam}_3, {v}_3)$ of $A$.
Then we use EIGSOL again to compute the lowest eigenpair 
$(\fl{\lam}_3,\fl{v}_3)$ of $\fl{A}_2$ satisfying 
\begin{align*}
\fl{A}_2 \fl{v}_3 = \fl{\lam}_3\fl{v}_3 + \eta_3,
\end{align*}
where 
the residual vector $\eta_3$ satisfies \eqref{eq:toldef}. 
Meanwhile, expressing the computed eigenpairs 
$(\fl{\lam}_1,\fl{v}_1)$ and $(\fl{\lam}_2,\fl{v}_2)$ 
in terms of $\fl{A}_2$, we have
\begin{align*}  
\fl{A}_2\fl{V}_2 =
\fl{V}_2(\fl{\Lam}_2 + \Sigma_2) + \fl{V}_2\Sigma_2\Phi_2 + E_2,
\end{align*}
where $\fl{\Lam}_2 = \diag(\fl{\lam}_1,\fl{\lam}_2)$,
$E_2 = [\eta_1,\eta_2]$, 
and $\Phi_2\in\R^{2\times 2}$ is the strictly lower triangular part of
the matrix $\fl{V}_2^\T\fl{V}_2 - I_2$, i.e., 
$\Phi_2 + \Phi^\T_2 =  \fl{V}_2^\T\fl{V}_2 - I_2$.

In general, at the $j$-th EED step, 
we choose a shift $\sigma_j$ and define
\begin{align} \label{eq:defAj} 
\fl{A}_{j} \equiv \fl{A}_{j-1} + \sigma_{j}\fl{v}_{j}\fl{v}_{j}^\T 
        = A + \fl{V}_{j}\Sigma_{j}\fl{V}_{j}^\T,
\end{align}
where $\fl{V}_{j}\equiv[\fl{v}_1,\ldots,\fl{v}_{j}]$
and $\Sigma_j=\diag(\sigma_1,\ldots,\sigma_j)$ with
$\fl{A}_0\equiv A$.
Then by choosing the shift $\sigma_j > \lambda_{\rm upper} - \fl{\lam}_j$,  
the lowest eigenpair of $\fl{A}_j$ 
is an approximation of the $(j+1)$-th eigenpair $({\lam}_{j+1}, {v}_{j+1})$ of $A$.
We use EIGSOL to compute the lowest eigenpair 
$(\fl{\lam}_{j+1}, \fl{v}_{j+1})$ of $\fl{A}_j$ satisfying 
\begin{align}  \label{eq:gov2b}
\fl{A}_j\fl{v}_{j+1} = \fl{\lam}_{j+1}\fl{v}_{j+1} + \eta_{j+1},
\end{align}
where $\|\fl{v}_{j+1}\|_2 = 1$ and 
the residual vector $\eta_{j+1}$ satisfies \eqref{eq:toldef}, i.e.,
\begin{align}  \label{eq:jtol}
\|\eta_{j+1}\|_2 \leq tol\cdot \|A\|_2.
\end{align} 
Meanwhile, for the computed eigenpairs $(\fl{\lam}_j,\fl{v}_j)$ 
in terms of $\fl{A}_j$, we have 
\begin{align} \label{eq:gov0}
\fl{A}_j\fl{v}_j = \fl{A}_{j-1}\fl{v}_{j} + \sigma_j\fl{v}_j
= (\lam_j + \sigma_j)\fl{v}_j + \eta_j,
\end{align}
and for the computed eigenpairs $(\fl{\lam}_i,\fl{v}_i)$
with $1\leq i\leq j-1$ in terms of $\fl{A}_j$, we have 
\begin{align}
\fl{A}_j\fl{v}_{i} 
& = \left( \fl{A}_{j-1} + \sigma_{j}\fl{v}_{j}\fl{v}_{j}^\T\right) 
\fl{v}_{i} \nonumber \\ 
& = \left(\fl{A}_{i-1} + \sigma_i\fl{v}_{i}\fl{v}_i^\T 
+ \fl{V}_{i+1:j}\Sigma_{i+1:j}\fl{V}_{i+1:j}^\T \right)\fl{v}_{i} \nonumber \\
& = \left(\fl{A}_{i-1} + \sigma_i\fl{v}_{i}\fl{v}_i^\T \right)\fl{v}_i  
+ \fl{V}_{i+1:j}\Sigma_{i+1:j}\fl{V}_{i+1:j}^\T\fl{v}_{i} \nonumber \\
& = \fl{\lam}_{i}\fl{v}_{i} + \eta_{i} + \sigma_i\fl{v}_i 
        + \fl{V}_{i+1:j}\Sigma_{i+1:j}\fl{V}_{i+1:j}^\T\fl{v}_{i}  \nonumber \\
& = (\fl{\lam}_{i} + \sigma_i)\fl{v}_i + 
\fl{V}_{i+1:j}\Sigma_{i+1:j}\fl{V}_{i+1:j}^\T\fl{v}_{i} + \eta_i  \nonumber \\
& = (\fl{\lam}_{i} + \sigma_i)\fl{v}_i + \fl{V}_j\Sigma_j\twobyone{0}{\fl{V}_{i+1:j}^\T\fl{v}_{i}} + \eta_i,
\label{eq:gov1}
\end{align}
where $\fl{V}_{i+1:j} \equiv [\fl{v}_{i+1},\ldots,\fl{v}_j]$
and $\Sigma_{i+1:j} \equiv \diag(\sigma_{i+1},\ldots,\sigma_j)$.

Combining \eqref{eq:gov0} and \eqref{eq:gov1}, we have
\begin{align}  \label{eq:gov2a}
\fl{A}_j\fl{V}_j = 
\fl{V}_j(\fl{\Lam}_j + \Sigma_j) + \fl{V}_j\Sigma_j\Phi_j + E_j,
\end{align}
where 
$\fl{\Lam}_j=\diag(\fl{\lam}_1,\ldots,\fl{\lam}_j)$, 
$E_j = [\eta_1,\ldots,\eta_j]$,  $\Phi_j$ is the strictly lower 
triangular part of the matrix $\fl{V}_j^\T\fl{V}_j - I_j$ and 
$\Phi_j + \Phi^\T_j =  \fl{V}_j^\T\fl{V}_j - I_j$.
Eqs.~\eqref{eq:gov2b} and~\eqref{eq:gov2a} are referred to as
{\em the governing equations of $j$ steps of (inexact) EED procedure.}


For the backward stability analysis of the EED procedure, 
we introduce the following two quantities associated with the shifts
$\sigma_1,\dots,\sigma_j$ for a $j$-step EED: 
\begin{itemize} 
\item the {\em spectral gap} of $\fl{A}_j$, 
defined as the separation between the computed eigenvalues and the
shifted ones: 
\begin{equation}  \label{eq:gapdef}
\gap_j \equiv \min_{\lam\in\mc{I}_{j+1}, \theta\in\mc{J}_{j}} 
| \lam - \theta | > 0, 
\end{equation}
where 
$\mc{I}_{j+1} \equiv \{\fl{\lam}_1,\ldots, \fl{\lam}_{j}, \fl{\lam}_{j+1}\}$, 
the set of computed eigenvalues, and 
$\mc{J}_{j} \equiv \{\fl{\lam}_1+\sigma_1,\ldots,\fl{\lam}_{j}+\sigma_j\}$,
the set of computed eigenvalues after shifting
(see Figure~\ref{fig:gap} for an illustration);

\begin{figure}
\centering
\begin{tikzpicture}
\begin{axis}[
	width=0.80\textwidth,
	height=0.2\textwidth,
	hide y axis,
	axis lines = center,
	xmin = -3,
	xmax = 3,
	ymin = -1.0,
	ymax = 1.0,
	xtick = {-2.6,-2.4,-1.8,-1.7,-1.65,-1.3,1.0,1.05,1.8,2.1,2.15,2.6},
	xticklabels = {$\fl{\lam}_{1}$, ,$\fl{\lam}_i$, , ,$\fl{\lam}_{j+1}$,
			     $\fl{\lam}_1+\sig_1$, ,$\fl{\lam}_i+\sig_i$, , ,$\fl{\lam}_{j}+\sig_j$},
	xticklabel style={text height=2ex},
	major tick length=0.2cm,
	every tick/.style={
		black,
		thick,
	}
]
\draw[<->] (-1.27,0.3) -- (-0.15,0.3) node[above] {$\gap_j$} -- (0.97,0.3);
\end{axis}
\end{tikzpicture}
\caption{Illustration of the spectral gap $\gap_j$} \label{fig:gap} 
\end{figure}
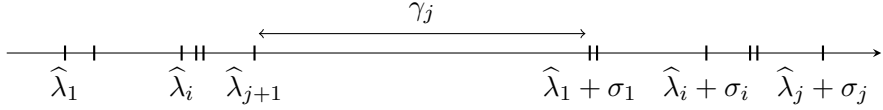

\item the {\em shift-gap ratio}, defined as the ratio 
of the largest shift 
and the spectral gap $\gap_j$:
\begin{equation}  \label{eq:taudef}
\tau_j \equiv 
\frac{1}{\gap_j} \cdot \max_{1\leq i\leq j} |\sigma_i|.
\end{equation}
\end{itemize}
We will see that $\gap_j$ and $\tau_j$ are crucial quantities to characterize
the backward stability of the EED procedure.

%
%
\section{Backward stability analysis of EED}  \label{sec:backErr}

In this section, we first review the notion of backward stability for
the computed eigenpairs. Then we derive upper bounds 
on the loss of orthogonality of computed eigenvectors 
and the backward error norm of computed eigenpairs
by the EED procedure, and reveal conditions for the backward stability of the
procedure. 

\subsection{Notion of backward stability} 
By the well-established notion of backward stability analysis
in numerical linear algebra \cite{Higham1998,Sun1995}, 
the following two quantities measure the accuracy of 
the approximate eigenpairs $(\fl{\Lam}_{j+1}, \fl{V}_{j+1})$ of $A$
computed by the EED procedure:
\begin{itemize}
\item 
the loss of orthogonality of the computed eigenvectors $\fl{V}_{j+1}$,
\begin{align}  \label{eq:omgdef}
\omega_{j+1}\equiv\|\fl{V}_{j+1}^\T\fl{V}_{j+1} - I_{j+1}\|_{\rm F},
\end{align}

\item 
the symmetric backward error norm of the computed
eigenpairs $(\fl{\Lam}_{j+1},\fl{V}_{j+1})$,
\begin{align} \label{eq:deltadef} 
\delta_{j+1}\equiv\min_{\Delta\in\mc{H}_{Q_{j+1}}} \|\Delta\|_\F,
\end{align}
where $\mc{H}_{Q_{j+1}}$ is the set of the symmetric backward errors
for the orthonormal basis $Q_{j+1}$ from the polar decomposition\footnote{The
polar decomposition of a matrix $X$ is given by $X=UP$ where $U$ is
orthonormal and $P$ is symmetric positive semi-definite.}
of $\fl{V}_{j+1}$, namely,
\begin{align}  \label{eq:Hdef}
\mc{H}_{Q_{j+1}} \equiv 
        \left\{
        \Delta \mid 
        (A + \Delta)Q_{j+1} = Q_{j+1}\fl{\Lam}_{j+1},\ 
        \Delta = \Delta^\T \in \R^{n\times n}
        \right\}.
\end{align}
\end{itemize}

For a prescribed tolerance $tol$ of 
the stopping criterion \eqref{eq:jtol} for an eigensolver EIGSOL,
the EED procedure is considered to be {\em backward stable} if 
\begin{align}  \label{eq:goal1}
\omega_{j+1} = O(tol)
\end{align}
and
\begin{align}  \label{eq:goal2}
\delta_{j+1} = O(tol\cdot\|A\|_2),
\end{align}
where the constants in the big-O notations
are low-degree polynomials in the number $j$ of the EED steps.

\subsection{Loss of orthogonality}\label{subsec:lossorth} 

%
We first prove the following lemma to reveal the structure 
of the orthogonality between the computed eigenvectors.
%
%
\begin{lemma} \label{lemma:cosine}
By the governing equations 
\eqref{eq:gov2b} and \eqref{eq:gov2a} of $j$ steps of EED,
if $\tau_j \omega_j < \sqrt{2}$, then for $i=1,2,\dots, j$, 
the matrices $\Gamma_i \equiv \fl{\Lam}_i + \Sigma_i - \fl{\lam}_{i+1} I_i$ 
and $I_i + \Phi_i^\T\Sigma_i\Gamma_i^{-1}$ are non-singular, and 
\begin{align}  \label{eq:cosine}
\fl{V}_i^\T\fl{v}_{i+1} =
\Gamma_i^{-1}\left(I_i + \Phi_i^\T\Sigma_i\Gamma_i^{-1}\right)^{-1}
\left[ \fl{V}_{i}^\T\eta_{i+1} - E_i^\T\fl{v}_{i+1} \right].
\end{align}
Furthermore,
\begin{enumerate}[(i)]  
\item \label{eq:gammabnd} 
$\|\Gamma_i^{-1}\|_2 \leq \gamma_{j}^{-1}$,

\item \label{eq:matplusbnd}
$\|(I_i + \Phi_i^\T\Sigma_i\Gamma_i^{-1})^{-1}\|_2 
\leq (1 - \tau_j\omega_j/\sqrt{2})^{-1}$,
\end{enumerate}
where  $\gap_j$ and $\tau_j$ are the spectral gap and
the shift-gap ratio defined in \eqref{eq:gapdef} and \eqref{eq:taudef}, 
respectively, and $\omega_j$ is the loss of the orthogonality defined 
in \eqref{eq:omgdef}.  
\end{lemma}
%
%
\begin{proof} 
By the governing equations 
\eqref{eq:gov2b} and \eqref{eq:gov2a} of $j$ steps of EED, 
for $1 \leq i \leq j$, we have  
\begin{align*}
\fl{V}_i^\T\fl{A}_i\fl{v}_{i+1} 
= \fl{\lam}_{i+1}\fl{V}_i^\T\fl{v}_{i+1} + \fl{V}_i^\T\eta_{i+1}
\end{align*} 
and 
\begin{align*}
\fl{V}_i^\T\fl{A}_i\fl{v}_{i+1} 
= (\fl{\Lam}_i + \Sigma_i)\fl{V}_i^\T\fl{v}_{i+1} 
+ \Phi_i^\T\Sigma_i\fl{V}_i^\T\fl{v}_{i+1} + E_i^\T\fl{v}_{i+1}.
\end{align*}
Consequently,
\begin{align}  \label{eq:gov4}
(\Gamma_i + \Phi_i^\T\Sigma_i)\fl{V}_i^\T\fl{v}_{i+1}
= \fl{V}_{i}^\T\eta_{i+1} - E_i^\T\fl{v}_{i+1},
\end{align}
where $\Gamma_i = \fl{\Lam}_i + \Sigma_i - \fl{\lam}_{i+1} I_i$ is 
a diagonal matrix.

By the definition \eqref{eq:gapdef} of the spectral gap $\gap_j$,
\begin{align}  \label{eq:(i)}
\sigma_{\min}(\Gamma_i) = 
\min_{1\leq k\leq i} 
|\fl{\lam}_{k} + \sigma_{k} - \fl{\lam}_{i+1}| \geq \gap_j > 0.
\end{align}
Hence the matrix $\Gamma_i$ is non-singular
and the bound \eqref{eq:gammabnd} holds. 

Since $\Phi_i$ is the strictly lower triangular 
part of the matrix $\fl{V}_i^\T\fl{V}_i - I_i$,
\begin{align*}
\|\Phi_i^\T\|_2\leq\|\Phi_i^\T\|_{\rm F} 
	= \frac{\omega_i}{\sqrt{2}} \leq \frac{\omega_j}{\sqrt{2}}.
\end{align*}
Consequently, 
\begin{align}  \label{eq:gaperr}
\|\Phi_i^\T\Sigma_i\Gamma_i^{-1}\|_2 
	\leq \|\Phi_i^\T\|_2\|\Gamma_i^{-1}\|_2\|\Sigma_i\|_2 
    \leq \frac{\omega_j}{\sqrt{2}} \cdot \gamma_j^{-1}\cdot {\|\Sigma_j\|_2}
	= \frac{\omega_j}{\sqrt{2}}\cdot\tau_j < 1, 
\end{align}
where for the last inequality, we use the assumption  $\tau_j \omega_j < \sqrt{2}$.
By \eqref{eq:gaperr}, the matrix $I_i + \Phi_i^\T\Sigma_i\Gamma_i^{-1}$ 
is non-singular and the bound \eqref{eq:matplusbnd} holds due to 
$\|(I + X)^{-1}\|_2 \leq (1 - \|X\|_2)^{-1}$ for any $X$ with $\|X\|_2<1$.

Since both matrices $\Gamma_i$ and 
$I_i + \Phi_i^\T\Sigma_i\Gamma_i^{-1}$ 
are invertible, the identity \eqref{eq:cosine} follows from 
\eqref{eq:gov4}.
\end{proof}

%
%
Next we exploit the structure of the product $\fl{V}_i^\T\fl{v}_{i+1}$ 
to derive a computable upper bound 
on the loss of orthogonality $\omega_{j+1}$ 
of computed eigenvectors $\fl{V}_{j+1}$.

\begin{theorem} \label{thm:omg}
By the governing equations
\eqref{eq:gov2b} and \eqref{eq:gov2a} of $j$ steps of EED,
if $\tau_j \omega_j < \sqrt{2}$, then the loss of orthogonality 
$\omega_{j+1}$ of the computed
eigenvectors $\fl{V}_{j+1}$ defined in \eqref{eq:omgdef} satisfies 
\begin{align}  \label{eq:omgbnd}
\omega_{j+1} \leq 2\, 
\frac{c_j}{\gap_j} \left(1 + 2\frac{c_j}{\gap_j}\|E_{j+1}\|_{\rm F} \right)
\|E_{j+1}\|_{\rm F}, 
\end{align}
where $c_j = (1 - \tau_j\omega_j/\sqrt{2})^{-1}$,
and $\gap_j$ and $\tau_j$ are the spectral gap and
the shift-gap ratio defined in \eqref{eq:gapdef} and \eqref{eq:taudef},
respectively.
\end{theorem}
%
%
\begin{proof} 
By the definition \eqref{eq:omgdef}, we have 
\begin{align} \label{eq:omgcos}
\omega_{j+1}^2 
= 2\cdot\left\| \Phi^\T_{j+1} \right\|^2_\F
= 2\cdot\sum_{i=1}^{j}\|\fl{V}_i^\T\fl{v}_{i+1}\|_2^2.
\end{align}
Recall Lemma \ref{lemma:cosine} that, for any $1 \leq i \leq j$,
\begin{align*}
\|\fl{V}_i^\T\fl{v}_{i+1}\|_2
\leq \frac{c_j}{\gap_j}\cdot\|\fl{V}_{i}^\T\eta_{i+1} - E_i^\T\fl{v}_{i+1}\|_2.
\end{align*}
Hence we can derive from \eqref{eq:omgcos} that
\begin{align} 
\omega_{j+1}^2 
&\leq \frac{2c^2_j}{\gap_j^2}\cdot\sum_{i=1}^{j}\|\fl{V}_i^\T\eta_{i+1} -
E_i^\T\fl{v}_{i+1}\|_2^2 \notag \\
&= \frac{2c^2_j}{\gap_j^2}\cdot \frac{1}{2}\|\fl{V}_{j+1}^\T E_{j+1} - E_{j+1}^\T\fl{V}_{j+1}\|_{\rm F}^2 \nonumber \\
& \leq \frac{2c^2_j}{\gap_j^2}\cdot 2\|\fl{V}_{j+1}^\T E_{j+1}\|^2_{\F} 
\ \leq\ \frac{4c^2_j}{\gap_j^2}\cdot \|\fl{V}_{j+1}^\T\|^2_2 \|E_{j+1}\|^2_{\F}.   \label{eq:key}
\end{align}
Since 
\begin{align*}
\|\fl{V}_{j+1}^\T\|^2_2 
= \|\fl{V}_{j+1}^\T\fl{V}_{j+1}\|_2
\leq \|I_{j+1}\|_2 + \|\fl{V}_{j+1}^\T\fl{V}_{j+1} - I_{j+1}\|_2 
\leq 1 + \omega_{j+1},
\end{align*}
we arrive at 
\begin{align} \label{eq:delbd3}
\omega_{j+1}^2 \leq
\frac{4c_j^2}{\gap_j^2}\cdot(1+\omega_{j+1})\cdot\|E_{j+1}\|_{\rm F}^2.
\end{align}
Letting $t=\omega_{j+1}/{\chi}_{j+1}$, where
${\chi}_{j+1} = 2c_j\|E_{j+1}\|_{\rm F}/\gap_j$,
then the inequality \eqref{eq:delbd3} is recast as
\begin{align}  \label{eq:quadratic}
t^2 - {\chi}_{j+1} t - 1\leq 0.
\end{align}
By the fact that the quadratic polynomial in \eqref{eq:quadratic} 
is concave, we conclude that
\begin{align*}
t \leq \frac{1}{2}\cdot\left({\chi}_{j+1} + \sqrt{4 + {\chi}_{j+1}^2}\right)
\leq \frac{1}{2}\cdot\left({\chi}_{j+1} + 2 + {\chi}_{j+1} \right)
\leq 1 + {\chi}_{j+1}.
\end{align*}
This proves the upper bound in \eqref{eq:omgbnd}.
\end{proof}

\begin{remark}{\rm 
A straightforward way to derive an upper bound of $\omega_{j+1}$ 
is to use the relation 
\begin{equation*}
\omega_{j+1}^2 = \omega_j^2 + 2\cdot\|\fl{V}_j^\T\fl{v}_{j+1}\|_2^2,
\end{equation*}
and the following bound from Lemma~\ref{lemma:cosine}
\begin{align*}
\|\fl{V}_j^\T\fl{v}_{j+1}\|_2^2 
 \leq \frac{c_j^2}{\gap_j^2} \cdot 
       \left( \|\fl{V}_j\|_2\|\eta_{j+1}\|_2 + \|E_j\|_\F \right)^2
\approx \frac{c_j^2}{\gap_j^2} (\|E_{j+1}\|_\F^2 + 2\|\eta_{j+1}\|_2\|E_j\|_\F).
\end{align*}
However, this would lead to a pessimistic upper bound of $\omega_{j+1}$. 
In contrast, in Theorem~\ref{thm:omg}, 
we exploit the structure of the orthogonality $\fl{V}_i^\T\fl{v}_{i+1}$ 
to obtain a tighter bound \eqref{eq:omgbnd} of $\omega_{j+1}$. 
Eq.~\eqref{eq:key} in the proof is the key step. 
In Example~\ref{eg:eg1} in Section~\ref{sec:examples},
we will demonstrate numerically that the bound \eqref{eq:omgbnd} 
is indeed a tight upper bound on the loss of orthogonality $\omega_{j+1}$.
}\end{remark}

%
%
\subsection{Symmetric backward error norm}
\label{subsec:backErr} 

In this section, we derive 
a computable upper bound on the symmetric backward error norm $\delta_{j+1}$
of computed eigenpairs $(\fl{\Lam}_{j+1}, \fl{V}_{j+1})$ of $A$ 
defined in \eqref{eq:deltadef}.  
First, the following lemma gives an upper bound of the 
norm of the residual for $(\fl{\Lam}_{j+1}, \fl{V}_{j+1})$: 
\begin{align} \label{eq:resdef} 
R_{j+1} \equiv A\fl{V}_{j+1} - \fl{V}_{j+1}\fl{\Lam}_{j+1}.
\end{align}

%
%
\begin{lemma}  \label{lemma:Res}
By the governing equations
\eqref{eq:gov2b} and \eqref{eq:gov2a} of $j$ steps of EED,
if $\tau_j \omega_j < \sqrt{2}$,
then for the computed eigenpairs $(\fl{\Lam}_{j+1}, \fl{V}_{j+1})$ of $A$, 
the Frobenius norm of the residual $R_{j+1}$ defined 
in \eqref{eq:resdef} satisfies 
\begin{align}  \label{eq:resbnd}
\|R_{j+1}\|_\F \leq
\left( 1 + \sqrt{2} c_j\tau_j (1 + \omega_{j+1})\right)
 \|E_{j+1}\|_\F,
\end{align}
where $c_j = (1 - \tau_j\omega_j/\sqrt{2})^{-1}$, 
and $\gap_j$ and $\tau_j$ are the spectral gap and
the shift-gap ratio defined in \eqref{eq:gapdef} and \eqref{eq:taudef},
respectively.
\end{lemma}
\begin{proof} ~ 
From the governing equation \eqref{eq:gov2a} of the EED procedure 
after $j+1$ steps, we have
\begin{align}  
\fl{A}_{j+1}\fl{V}_{j+1} = \fl{V}_{j+1}(\fl{\Lam}_{j+1} + \Sigma_{j+1}) 
	+ \fl{V}_{j+1}\Sigma_{j+1}\Phi_{j+1} + E_{j+1}. \label{eq:res_gov}
\end{align}
On the other hand, by the definition \eqref{eq:defAj} of
$\fl{A}_{j+1}$, we have
\begin{align}
\fl{A}_{j+1}\fl{V}_{j+1} 
& = A\fl{V}_{j+1} + \fl{V}_{j+1}\Sigma_{j+1}\fl{V}_{j+1}^\T\fl{V}_{j+1} 
     \nonumber \\
& = A\fl{V}_{j+1} + \fl{V}_{j+1}\Sigma_{j+1}
      (\Phi_{j+1} + I_{j+1} + \Phi_{j+1}^\T).  \label{eq:res_expand}
\end{align}
Combining \eqref{eq:res_gov} and \eqref{eq:res_expand}, 
we obtain the residual
\begin{align}
R_{j+1} = A\fl{V}_{j+1} - \fl{V}_{j+1}\fl{\Lam}_{j+1} 
= E_{j+1} - \fl{V}_{j+1}\Sigma_{j+1}\Phi_{j+1}^\T  \label{eq:Res}.
\end{align}
Consequently, the norm of the residual $R_{j+1}$ is bounded by
\begin{align} 
\|R_{j+1}\|_{\rm F} \leq \|E_{j+1}\|_{\rm F} + 
  \|\fl{V}_{j+1}\|_2\|\Sigma_{j+1}\Phi_{j+1}^\T\|_{\rm F}.  \label{eq:RFro1}
\end{align}
Note that $\Sigma_{j+1} = \diag(\sigma_1,\ldots,\sigma_{j+1})$
and $\Phi_{j+1}^\T$ is the strictly upper triangular part of 
the matrix $\fl{V}_{j+1}^\T\fl{V}_{j+1} - I_{j+1}$, and we have
\begin{align} 
\|\fl{V}_{j+1}\|_2\|\Sigma_{j+1}\Phi_{j+1}^\T\|_\F
& \leq \|\fl{V}_{j+1}\|_2\|\Sigma_j\|_2\|\Phi_{j+1}^\T\|_\F \nonumber \\
& \leq \|\fl{V}_{j+1}\|_2\|\Sigma_j\|_2\cdot\frac{1}{\sqrt{2}}\omega_{j+1} \nonumber \\
& \leq \frac{1}{\sqrt{2}}\|\Sigma_j\|_2\sqrt{(1+\omega_{j+1})\omega_{j+1}^2},  
\label{eq:SigTheFro2}
\end{align} 
where, for the third inequality, we again use the fact
that $\|\fl{V}_{j+1}\|_2 \leq \sqrt{1+\omega_{j+1}}$  by the definition of 
the loss of orthogonality $\omega_{j+1}$. 
Left-multiplying \eqref{eq:delbd3} by $1+\omega_{j+1}$,
we know that
\begin{align*}  
(1 + \omega_{j+1})\omega_{j+1}^2 
\leq \frac{4c_j^2}{\gap_j^2}\cdot (1+\omega_{j+1})^2\cdot \|E_{j+1}\|_\F^2.
\end{align*}
Plugging into \eqref{eq:SigTheFro2}, we obtain  
\begin{align*}  \label{eq:V2SigTheFro}
\|\fl{V}_{j+1}\|_2\|\Sigma_{j+1}\Phi_{j+1}^\T\|_\F
& \leq \sqrt{2}\|\Sigma_j\|_2c_j\gap_j^{-1}\cdot(1+\omega_{j+1})\cdot\|E_{j+1}\|_\F \nonumber \\
& = \sqrt{2}c_j\tau_j\cdot(1+\omega_{j+1})\cdot\|E_{j+1}\|_\F.
\end{align*}
%
%
Combine with \eqref{eq:RFro1} and we arrive at the upper bound \eqref{eq:resbnd}
of $\|R_{j+1}\|_{\rm F}$.
\end{proof}

Now we recall the following theorem by Sun \cite{Sun1995}.
\begin{theorem}[Thm.3.1 in \cite{Sun1995}] \label{thm:Sun1995}
Let $(\Lambda,X)$ be approximate eigenpairs of a symmetric matrix $A$.
Let $Q$ be the orthonormal basis from the polar decomposition of $X$.
Define the set 
\begin{align}
\mc{H}_{Q} = 
\{\Delta \mid (A+\Delta)Q = Q\Lambda,\ \Delta=\Delta^\T \in\R^{n\times n} \}.
\end{align}
Then $\mc{H}_{Q}$ is non-empty and 
there exists a unique $\Delta_{Q}\in\mc{H}_{Q}$ such that
\begin{align}  \label{eq:Sun1995}
\min_{\Delta\in\mc{H}_{Q}} \|\Delta\|_{\rm F} = 
\|\Delta_{Q}\|_{\rm F}  \leq
\sqrt{\|R\|^2_{\rm F} + \|\mc{P}_{X}^{\perp}R\|^2_{\rm
F}}\Big/{\sigma_{\min}(X)} ,
\end{align}
where $R = AX - X\Lambda$ is the residual
and $\mc{P}_{X}^{\perp}$ is the orthogonal projection
onto the orthogonal complement of $\mc{R}(X)$. 
\end{theorem}


%
%
From Lemma \ref{lemma:Res} and Theorem \ref{thm:Sun1995},
we have the following computable upper bound on the symmetric backward 
error norm $\delta_{j+1}$ of the computed eigenpairs 
$(\fl{\Lam}_{j+1}, \fl{V}_{j+1})$ of $A$. 

\begin{theorem}  \label{thm:backErr}
By the governing equations
\eqref{eq:gov2b} and \eqref{eq:gov2a} of $j$ steps of EED,
if $\tau_j \omega_j < \sqrt{2}$ and $\omega_{j+1} < 1$,
then the symmetric backward error norm $\delta_{j+1}$ 
of the computed eigenpairs $(\fl{\Lam}_{j+1}, \fl{V}_{j+1})$ of $A$
defined in \eqref{eq:deltadef} has the following upper bound
\begin{align}  \label{eq:backErrBnd}
\delta_{j+1} \leq 
\sqrt{2}\left( \frac{1+ c_j\tau_j (1 + \omega_{j+1}) }
                    {\sqrt{1-\omega_{j+1}}} \right) \|E_{j+1}\|_{\rm F}, 
\end{align}
where $c_j = (1 - \tau_j\omega_j/\sqrt{2})^{-1}$,
and $\gap_j$ and $\tau_j$ are the spectral gap and
the shift-gap ratio defined in \eqref{eq:gapdef} and \eqref{eq:taudef},
respectively.
\end{theorem}
\begin{proof} 
For the computed eigenvectors $\fl{V}_{j+1}$ of $A$, 
let $Q_{j+1}$ be the orthonormal basis from the 
polar decomposition of $\fl{V}_{j+1}$, and
the set $\mc{H}_{Q_{j+1}}$ be defined in \eqref{eq:Hdef}.
It follows from the definition~\eqref{eq:deltadef} and Theorem \ref{thm:Sun1995} that 
\begin{align} 
\delta_{j+1} & = \min_{\Delta\in\mc{H}_{Q_{j+1}}} \|\Delta\|_{\rm F}
 \leq \frac{1}{\sigma_{\min}(\fl{V}_{j+1})}
\sqrt{\|R_{j+1}\|^2_{\rm F} + \|\mc{P}^{\perp}R_{j+1}\|^2_{\rm F}}, 
\label{eq:Sun1995a}
\end{align}
where 
$R_{j+1}$ is the residual 
of $(\fl{\Lam}_{j+1}, \fl{V}_{j+1})$ 
defined in \eqref{eq:resdef}, 
and $\mc{P}^{\perp}$ is the orthogonal projection 
onto the orthogonal complement of the subspace $\mc{R}(\fl{V}_{j+1})$, i.e.,
\[ 
\mc{P}^{\perp} 
= I - \fl{V}_{j+1}(\fl{V}^T_{j+1} \fl{V}_{j+1})^{-1} \fl{V}^T_{j+1}.
\] 
By the equation~\eqref{eq:Res}, 
$\mc{P}^{\perp}R_{j+1} = \mc{P}^{\perp}E_{j+1}$. Hence we have 
\begin{align} \label{eq:PResbnd}
\|\mc{P}^{\perp}R_{j+1}\|_{\rm F} 
= \|\mc{P}^{\perp}E_{j+1}\|_{\rm F} \leq \|E_{j+1}\|_{\rm F}.
\end{align}
On the other hand, 
by the definition \eqref{eq:omgdef} of $\omega_{j+1}$ and 
the assumption $\omega_{j+1} < 1$, we have 
\begin{align*}  
|\sigma^2_{\min}(\fl{V}_{j+1}) - 1| 
\leq \|\fl{V}_{j+1}^\T\fl{V}_{j+1} - I_{j+1}\|_2 
\leq \omega_{j+1} < 1,
\end{align*}
which implies  the following lower bound of 
the singular value 
\begin{align}  \label{eq:sigvlower1}
\sigma_{\min}(\fl{V}_{j+1}) \geq \sqrt{1 - \omega_{j+1}}.
\end{align}
Plug \eqref{eq:PResbnd} and \eqref{eq:sigvlower1}
into \eqref{eq:Sun1995a}
and recall the upper bound of $\|R_{j+1}\|_\F$ in Lemma~\ref{lemma:Res}, 
and then we obtain
\begin{align*} 
\delta_{j+1} 
& \leq\sqrt{1+\left(1+\sqrt{2}c_j\tau_j\cdot(1+\omega_{j+1})\right)^2}
   \cdot\frac{\|E_{j+1}\|_{\rm F}}{\sqrt{1 - \omega_{j+1}}}  \nonumber \\
& \leq\sqrt{2}\left( 
     \frac{1 + c_j\tau_j (1+ \omega_{j+1})}{\sqrt{1 - \omega_{j+1}}}
     \right)\|E_{j+1}\|_{\rm F}, 
\end{align*}
where the second inequality is due to
$1+(1+\sqrt{2}a)^2 \leq 2(1+2a+a^2) = 2(1+a)^2$.
This completes the proof. 
\end{proof}

\begin{remark} \label{rmk:delres} 
{\rm
We note from the inequalities \eqref{eq:Sun1995a} 
and \eqref{eq:sigvlower1} that, when $\omega_{j+1} \ll 1$,
the residual norm $\|R_{j+1}\|_\F$ can be used 
as an easy-to-compute estimate of $\delta_{j+1}$. 
We will use this observation in numerical examples 
in Section \ref{sec:examples}.
}\end{remark}

%
%
\subsection{Backward stability of EED}  \label{sec:conditions}
In Theorems \ref{thm:omg} and \ref{thm:backErr}, the upper bounds
\eqref{eq:omgbnd} and \eqref{eq:backErrBnd} 
for $\delta_{j+1}$ and $\omega_{j+1}$ 
involve the quantity $\omega_{j}$ 
from the previous EED step. 
In this section, under a mild assumption, we derive explicit upper bounds for
$\omega_{j+1}$ and $\delta_{j+1}$, and then reveal 
conditions for the backward stability of the EED procedure.

\begin{lemma} \label{lem:wd}
Consider $j$ steps of EED  governed by Eqs.~\eqref{eq:gov2b}
and~\eqref{eq:gov2a}.  Assume 
\begin{align}  \label{eq:assumption0}
\tau_{j} \frac{\|A\|_2}{\gap_{j}}\cdot 4\sqrt{j+1}\cdot tol < 0.1.
\end{align}
Then 
\begin{enumerate}[(i)]
\item\label{lem:wd:a}
it holds that 
\begin{equation}\label{eq:tauiomega}
\tau_i\omega_i < 0.11
\quad\text{and}\quad
c_i = (1-\tau_i\omega_i/\sqrt{2})^{-1} < 2
\quad \text{for $i = 1,2,\dots, j$;}
\end{equation}

\item\label{lem:wd:b}
the loss of orthogonality $\omega_{j+1}$ is bounded by 
\begin{align}  \label{eq:omgbnd0} 
    \omega_{j+1} \le \left(\frac{\|A\|_2}{\gamma_j}\cdot 5
    \sqrt{j+1}\right)\cdot tol;
\end{align}

\item\label{lem:wd:c}
the backward error norm $\delta_{j+1}$ is bounded by 
\begin{equation}\label{eq:deltabnd0} 
    \delta_{j+1} \leq  
    \left(\tau_j \cdot 5 \sqrt{j+1}\right) \cdot tol\cdot \|A\|_2.
\end{equation}
\end{enumerate}
\end{lemma}
\begin{proof}
First observe that 
by the definitions~\eqref{eq:gapdef} and~\eqref{eq:taudef},
$\gamma_i$ is monotonically decreasing with the index $i$ and 
$\tau_i\geq 1$ is monotonically increasing with $i$.
Therefore, the assumption~\eqref{eq:assumption0} implies 
the inequalities 
\begin{align}  \label{eq:assumption12}
\frac{\|A\|_2}{\gap_i}\cdot 4\sqrt{i+1}\cdot tol < 0.1 
\quad\text{and}\quad 
\tau_{i}
\frac{\|A\|_2}{\gap_{i-1}}\cdot 4\sqrt{i}\cdot tol < 0.1
\quad
\text{for all $i\leq j$}.
\end{align}
Since the stopping criterion \eqref{eq:jtol} 
of EIGSOL implies 
\begin{equation}\label{eq:eip1}
\|E_{i}\|_\F = \| [\eta_1,\dots,\eta_{i}]\|_\F
\leq \sqrt{i}\cdot tol\cdot \|A\|_2, 
\end{equation}
inequalities~\eqref{eq:assumption12} leads to 
\begin{align}  \label{eq:assumption122}
\frac{4}{\gap_i}\cdot\|E_{i+1}\|_\F< 0.1
\quad\text{and}\quad 
\tau_{i}\cdot 
\frac{4}{\gap_{i-1}}\cdot \|E_i\|_\F < 0.1
\quad \text{for all $i\leq j$}.
\end{align}

\eqref{lem:wd:a} We prove the inequality \eqref{eq:tauiomega} by induction.
To begin with, recall that $\|\widehat v_1\|_2 =1$, which implies
$\omega_1 = \|\widehat v_1^T\widehat v_1 -1\|_F =0$,
$\tau_1\omega_1 =0< 0.11$, and $c_1 = 1 < 2$. 
Hence \eqref{eq:tauiomega} holds for $i=1$.
Now, for $2\leq i\leq j$, assume that $\tau_{i-1}\omega_{i-1} < 0.11$ and $c_{i-1} < 2$.
Since $\tau_{i-1}\omega_{i-1} < 0.11$,
we can apply Theorem~\ref{thm:omg}
and derive from \eqref{eq:omgbnd} that
\begin{equation}\label{eq:gw}
\tau_{i}\omega_{i}
 \leq \tau_{i}
\cdot  \frac{2c_{i-1}}{\gap_{i-1}}\|E_{i}\|_\F\cdot 
\left( 1 + \frac{2c_{i-1}}{\gap_{i-1}}\|E_{i}\|_\F \right) 
< 0.1 \cdot (1+0.1)  = 0.11,
\end{equation}
where the last inequality of \eqref{eq:gw} is by $2c_{i-1}<4$ and~\eqref{eq:assumption122}.
This implies immediately
\[
c_i = (1 - \tau_i\omega_i/\sqrt{2})^{-1} \leq (1 - 0.11/\sqrt{2})^{-1} < 2.
\]
Therefore,~\eqref{eq:tauiomega} follows by induction.

\eqref{lem:wd:b} 
Since we have $\tau_j\omega_j < 0.11$ and $c_j < 2$ by~\eqref{eq:tauiomega},
we can apply Theorem~\ref{thm:omg} and derive from \eqref{eq:omgbnd} that 
\begin{equation}\label{eq:omegajp1}
\omega_{j+1}
 \leq \frac{2c_{j}}{\gap_{j}}\cdot \|E_{j+1}\|_\F\cdot 
\left( 1 + \frac{2c_{j}}{\gap_{j}}\|E_{j+1}\|_\F \right) 
 \leq \frac{4}{\gap_{j}}\cdot \|E_{j+1}\|_\F \cdot
 \left( 1 + 0.1 \right), 
 \end{equation}
where in the second inequality we used $2c_j < 4$ 
and the first inequality in~\eqref{eq:assumption122}.
Recall the error bound of $\|E_{j+1}\|_\F$ from~\eqref{eq:eip1} 
and we obtain~\eqref{eq:omgbnd0}.

\eqref{lem:wd:c} 
We have $\tau_j\omega_j < 0.11$ and $c_j < 2$ by~\eqref{eq:tauiomega}.
It also follows from~\eqref{eq:omegajp1} and~\eqref{eq:assumption122}
that $\omega_{j+1} <0.11$.
Therefore, we can apply Theorem~\ref{thm:backErr} and 
derive from~\eqref{eq:backErrBnd} that
\[
\delta_{j+1}
\leq \sqrt{2}\left( \frac{1 + c_j\tau_j (1+ \omega_{j+1})}
     {\sqrt{1-\omega_{j+1}}} \right)\|E_{j+1}\|_{\F} 
\leq \sqrt{2}\left( \frac{1 + 2\tau_j(1+ 0.11)}
     {\sqrt{1-0.11}} \right)\cdot
     \|E_{j+1}\|_{\F},
 \]
 where in second inequality we used $0\leq \omega_{j+1}<0.11$.
 Since $\tau_j \geq 1$ by definition~\eqref{eq:taudef},
 we can relax the leading constant as 
 $\sqrt{2}(1.06 + 2.36 \cdot \tau_j) \leq 
 \sqrt{2} (3.42\cdot \tau_j) < 5 \tau_j$.
 Recall the error bound of $\|E_{j+1}\|_\F$ from~\eqref{eq:eip1} 
 and we prove~\eqref{eq:deltabnd0}.
\end{proof}

By the error bounds \eqref{eq:omgbnd0} and \eqref{eq:deltabnd0}
in Lemma~\ref{lem:wd}, we can see that the 
quantities $\gamma^{-1}_j\|A\|_2$ and $\tau_j$ play 
important roles for the stability of the EED procedure.
A sufficient condition to achieve the backward stability 
\eqref{eq:goal1} and \eqref{eq:goal2} is given by
$\gap^{-1}_j\|A\|_2 = O(1)$ and $\tau_j = O(1)$.  
In summary, we have the following theorem
for the backward stability of the EED procedure.

\begin{theorem} \label{thm:eedstable} 
Under the assumptions of   
the residual norm $\|\eta_{i}\|_{2}$ of EIGSOL 
satisfying \eqref{eq:toldef} and the inequality \eqref{eq:assumption0},
the backward stability of the EED procedure, 
in the sense of \eqref{eq:goal1} and \eqref{eq:goal2},
is guaranteed if the shifts $\sigma_1,\dots,\sigma_j$ are 
dynamically chosen such that 
\begin{equation} \label{eq:conditions}
\gap_j^{-1} \|A\|_2 = O(1)  
\quad\text{and}\quad
\tau_j = O(1). 
\end{equation}
\end{theorem} 

We note that when the shifts $\sigma_j$ are 
dynamically chosen such that the conditions \eqref{eq:conditions}
are satisfied, the assumption of 
the inequality ~\eqref{eq:assumption0} is indeed mild.

\begin{remark} \label{remark:instability}
{\rm 
From the upper bound \eqref{eq:omgbnd0} of 
the loss of orthogonality $\omega_{j+1}$,
we see that if the spectral gap $\gap_j$ is too small, 
i.e., $\gap_j \ll \|A\|_2$, then
$\omega_{j+1}$ could be amplified 
by a factor of $\gap_j^{-1}\|A\|_2$. 
On the other hand, 
from the upper bound \eqref{eq:deltabnd0} of 
the symmetric backward error norm $\delta_{j+1}$,
we see that when $\tau_j$ is too large, i.e., $\tau_j \gg 1$, 
$\delta_{j+1}$ could be amplified by a factor of $\tau_j$.
We will demonstrate these observations numerically in 
Example~\ref{eg:eg2} in Section~\ref{sec:examples}.
}\end{remark}

%
%
\section{EED in practice}  \label{sec:EEDpractice} 

In Section~\ref{sec:conditions},
we derived the required conditions \eqref{eq:conditions} 
of the spectral gap $\gap_j$ and the shift-gap ratio $\tau_j$
for the backward stability of the EED procedure.
The quantities $\gap_j$ and $\tau_j$ are controlled by 
the selection of the shifts $\sigma_1,\dots,\sigma_j$.  
In this section, we discuss a practical selection scheme of the shifts 
to satisfy the conditions \eqref{eq:conditions}, and then present
an algorithm that combines an eigensolver EIGSOL with the EED
procedure to compute eigenvalues in a prescribed interval
$\mathcal{I} = [\lam_{\rm low}, \lam_{\rm upper}]$.

\subsection{Choice of the shifts}  \label{sec:shifts}

We consider the following 
choice of the shift at the $j$-th EED step, 
\begin{align}  \label{eq:choice}
\sigma_j = \mu - \fl{\lam}_j,
\end{align}
where  $\mu\in\R$ is a parameter with $\mu > \lam_{\rm upper}$. 
The shifting scheme~\eqref{eq:choice} has been
used in several previous works, although without 
elaboration on the choice of parameter $\mu$; see 
\cite[Chap.5.1]{Parlett1998} and \cite{Money2005,Yamazaki2019}.
In the following, we will discuss how to choose the parameter 
$\mu$ such that the conditions~\eqref{eq:conditions} can hold.

To begin with, the choice of the shifts in~\eqref{eq:choice}  implies
that the spectral gap $\gap_j$ in \eqref{eq:gapdef} satisfies
\begin{align}  \label{eq:gapj_again}
\gap_j = \min_{ \theta\in\mc{I}_{j+1},\ \lam\in\mc{J}_{j} } |\lam-\theta|
	  = \min_{1\leq i\leq j+1} |\mu - \fl{\lam}_i|,
\end{align}
where 
$\mc{I}_{j+1} = \{\fl{\lam}_1,\ldots, \fl{\lam}_{j}, \fl{\lam}_{j+1}\}$
and $\mc{J}_{j} = \{\fl{\lam}_1+\sigma_1,\ldots,\fl{\lam}_{j}+\sigma_j\} = \{\mu\}$.
On the other hand, it also implies that the shift-gap ratio $\tau_j$ in \eqref{eq:taudef} 
satisfies 
\begin{align}  \label{eq:tauj_again}
\tau_j = \frac{1}{\gap_j} \cdot \max_{1\leq i\leq j} |\sigma_i| 
= \frac{\max_{1\leq i\leq j} |\mu - \fl{\lam}_i| }{\min_{1\leq i\leq j+1} |\mu - \fl{\lam}_i|}. 
\end{align} 
Now recall that $\mu > \lam_{\rm upper}$ and the computed eigenvalues 
$\fl{\lam}_i \in  [\lam_{\rm low}, \lam_{\rm upper}]$, for $i=1,2,\dots,j+1$,
so we have 
\begin{align*}
\mu - \lam_{\rm upper} 
\leq \min_{1\leq i\leq j+1} |\mu - \fl{\lam}_i| 
\leq\max_{1\leq i\leq j+1} |\mu - \fl{\lam}_i| 
\leq \mu - \lam_{\rm low}.
\end{align*}
Hence, \eqref{eq:gapj_again} and \eqref{eq:tauj_again} lead to 
\begin{align}  \label{eq:unibnds}
\gap_g \leq \gap_j \leq \gap_g\,\tau_g
\quad \mbox{and} \quad
\tau_j \leq \tau_g, 
\end{align} 
where 
\begin{equation*}  \label{eq:gparams}
\gap_g\equiv\mu - \lam_{\rm upper}
\quad \mbox{and} \quad 
\tau_g\equiv\frac{\mu - \lam_{\rm low}}{\mu - \lam_{\rm upper}}.
\end{equation*}
The problem is then turned to the choice of parameter $\mu$ 
such that the quantities $\gap_g$ and $\tau_g$ 
to satisfy the conditions~\eqref{eq:conditions}.

Let us consider a frequently encountered case in practice where
the width of the interval $\mathcal{I} = [\lam_{\rm low}, \lam_{\rm upper}]$ 
satisfies
$$
\lam_{\rm upper} - \lam_{\rm low} \leq \frac{1}{2}\|A\|_2.
$$  
Then by setting
\begin{equation*}  \label{eq:choice2}
\mu = \fl{\lam}_{1} + \|A\|_2,
\end{equation*} 
we have
\begin{equation*}  \label{eq:gpest}
\frac{1}{2}\|A\|_2 \leq 
\gap_g = 
\left( 1 - \frac{\lam_{\rm upper} - \fl{\lam}_1}{\|A\|_2} \right) 
\|A\|_2
\leq \|A\|_2 
\end{equation*} 
and 
\begin{equation*}  \label{eq:gpest2}
\tau_g = 1 + \frac{\lam_{\rm upper} - \lam_{\rm low}}{\gamma_g}
\leq 2.
\end{equation*}
Consequently, by \eqref{eq:unibnds},
the spectral gap $\gap_j$ and the shift-gap ratio 
$\tau_j$ satisfy the desired conditions \eqref{eq:conditions}.

In summary, for the EED procedure in practice,  
we recommend the use of the following strategy 
for the choice of the shift $\sigma_j$ 
at the $j$-th EED,
\begin{equation} \label{eq:shiftchoicefinal} 
\sigma_j = \mu - \fl{\lam}_j
\quad \mbox{with} \quad 
\mu = \fl{\lam}_{1} + \|A\|_2,
\end{equation} 
to compute the eigenvalues in an interval 
$\mc{I} = [\lam_{\rm low},\lam_{\rm upper}]$, 
where $\lam_{\rm upper} - \lam_{\rm low} \leq \frac{1}{2}\|A\|_2$.


\subsection{Algorithm}  \label{sec:algorithm}

An eigensolver EIGSOL combined with the EED procedure to 
compute the eigenvalues in an interval 
$\mc{I} = [\lam_{\rm low},\lam_{\rm upper}]$
at the lower end of the spectrum of $A$
is summarized in Algorithm \ref{alg:EIGSOL_EED}.

\begin{algorithm}[H]
\caption{
An eigensolver EIGSOL combined with the EED procedure
}
\label{alg:EIGSOL_EED}
\begin{algorithmic}[1] 
\REQUIRE{ the interval $\mc{I} = [\lam_{\rm low},\lam_{\rm upper}]$ 
at the lower end of the spectrum of $A$; the relative tolerance $tol$ 
in \eqref{eq:toldef} for EIGSOL.  
}
\ENSURE{ the approximate eigenpairs $(\fl{\lam}_i,\fl{v}_i)$ 
in the interval $\mc{I}$.
}

\STATE{$\fl{A}_0 = A$;}

\STATE{use EIGSOL to compute the lowest eigenpair 
$(\fl{\lam}_1,\fl{v}_1)$ of $\fl{A}_0$ and an estimate 
{\tt Anorm} of $\|A\|_2$;}

\STATE{$\mu = \fl{\lam}_{1} + {\tt Anorm}$;  \label{algline:setmu}}

\FOR{$j = 1, 2, \ldots$}

\STATE{$\sigma_j = \mu - \fl{\lam}_j$;  \label{algline:setshifts}}

\STATE{$\fl{A}_j = \fl{A}_{j-1} + \sigma_j\fl{v}_j\fl{v}_j^\T
 = A + \fl{V}_j\Sigma_j\fl{V}_j^\T$;
\label{algline:Ahatj} 
}

\STATE \label{algline:soldeflated}
{compute the lowest eigenpair $(\fl{\lam}_{j+1},\fl{v}_{j+1})$ 
of $\fl{A}_j$ by EIGSOL;}  

\STATE \label{algline:criterion}
{check if all the eigenpairs in the interval $\mc{I}$ have been computed;}

\ENDFOR

\STATE{return the approximate eigenpairs $(\fl{\lam}_i,\fl{v}_i)$ 
in the interval $\mc{I}$;}
\end{algorithmic}
\end{algorithm}

A few remarks are in order:
\begin{itemize}

\item In practice, we never need to form the matrix
$\fl{A}_j$ at step~\ref{algline:Ahatj} explicitly. 
We can assume that the only operation that is required by EIGSOL is the 
matrix-vector product $y:=\fl{A}_jx$. 

\item
At step \ref{algline:soldeflated},
the computation of the lowest eigenpair $(\fl{\lam}_{j+1},\fl{v}_{j+1})$ of $\fl{A}_j$
can be accelerated by warm starting the EIGSOL with
the lowest unconverged eigenvectors of $\fl{A}_{j-1}$.
This is possible for most iterative eigensolvers, 
such as TRLan \cite{Wu2000} and ARPACK \cite{Lehoucq1998}.  

\item
At step \ref{algline:criterion}, 
an ideal validation method is to 
use the inertias of the shifted matrix $A - \lam_{\rm upper} I$. 
However, computation of the inertias is a prohibitive cost 
for large matrices. An empirical validation is to monitor 
the lowest eigenvalue $\fl{\lam}_{j+1}$ of $\fl{A}_j$.
All eigenpairs in the interval $\mc{I}$ are considered 
to be found when $\fl{\lam}_{j+1}$ is outside the interval $\mc{I}$.
\end{itemize}

%
%
\section{Numerical experiments}  \label{sec:examples}

In this section, we first use synthetic examples to verify 
the sharpness of the upper bounds \eqref{eq:omgbnd} 
and \eqref{eq:backErrBnd} on the loss of orthogonality 
and the symmetric backward error norm of the EED procedure
under the choice \eqref{eq:shiftchoicefinal} of the shifts $\sigma_j$.
We present the cases where improper choices of the shifts $\sigma_j$ 
may lead to numerical instability of the EED procedure.
Then we demonstrate the numerical stability 
of the EED procedure for a set of large sparse symmetric matrices arising 
from applications.

\PARAGRAPH{(EIGSOL+EED) = (TRLan+EED).} 
We use TRLan as the eigensolver. 
TRLan is a C implementation of the thick-restart Lanczos 
method with adaptive sizes of the projection 
subspace~\cite{Wu2000,Yamazaki2010,Yamazaki2008}. 
The convergence criterion of an approximate 
eigenpair $(\fl{\lam}_{j+1},\fl{v}_{j+1})$ is the residual norm satisfying
%
%
\[ 
\|\eta_{j+1}\|_2  
=\|\widehat{A}_{j} \fl{v}_{j+1} - \fl{\lam}_{j+1}\fl{v}_{j+1} \|_2 
< tol\cdot {\tt Anorm},
\]
where $tol$ is a user-specified tolerance
and {\tt Anorm} is a $2$-norm estimate of $A$ computed by TRLan.
The starting vector is a random vector.



\begin{example} \label{eg:eg1} 
{\rm  
In this example, 
we demonstrate the sharpness of the upper bounds \eqref{eq:omgbnd} 
and \eqref{eq:backErrBnd} on the loss of orthogonality
and the symmetric backward error norm with the choice \eqref{eq:shiftchoicefinal}
of the shifts $\sigma_j$.

We consider a diagonal matrix $A$ with diagonal elements 
\begin{align*}
a_{kk} = 
\begin{cases}
\frac{1}{2}d_{k},\quad\text{if $1\leq k\leq n/2 $}, \\
\frac{1}{2}(1 + d_{k-n/2 }),\quad\text{if $n/2 < k \leq n$},
\end{cases}
\end{align*}
where $d_{k} = 10^{-5(1-\frac{k-1}{n/2 - 1})}$ and 
the matrix size $n=500$. The spectrum range of $A$ is $(0,1]$. 
The eigenvalues of $A$ are clustered around $0$ and $0.5$.
We are interested in computing the $n_e = 65$ eigenvalues in the interval 
$\mc{I} =  [0,\, 10^{-4}]$. 
The computed $2$-norm of $A$ is ${\tt Anorm} = 1.00$.

To closely observe the convergence, 
TRLan is slightly modified so that
the convergence test is performed at each Lanczos iteration.
The maximal dimension $m$ of the projection subspace is set to be $40$.

Numerical results of TRLan with the EED procedure
for computing all the $n_e$ eigenvalues in the interval $\mc{I}$
are summarized in Table~\ref{tab:stable}, where
the 4th column is the loss of orthogonality $\omega_{n_e}$,
the 5th column is the upper bound \eqref{eq:omgbnd} of $\omega_{n_e}$,
the 6th column is the residual norm $\|R_{n_e}\|_\F$
and the 7th column is the upper bound \eqref{eq:backErrBnd} of $\delta_{n_e}$.
Note that here we use the quantities $\delta_{n_e}$ and $\|R_{n_e}\|_\F$ 
interchangeably as discussed in Remark~\ref{rmk:delres}.

From Table~\ref{tab:stable}, we observe that 
with the choice \eqref{eq:shiftchoicefinal} of the shifts $\sigma_j$, 
$\gap^{-1}_g {\tt Anorm} \approx 1$ and $\tau_g \approx 1$.
Therefore, the conditions \eqref{eq:conditions} 
of the spectral gap $\gap_j$ and the shift-gap ratio $\tau_j$
for the backward stability are satisfied.
Consequently, the loss of orthogonality of the computed eigenvectors 
is $\omega_{n_e} = O(tol)$ and
the symmetric backward error norm of the 
computed eigenpairs $(\fl{\Lam}_{n_e},\fl{V}_{n_e})$ is 
$\delta_{n_e} = O(tol\cdot{\tt Anorm})$.
In addition, we observe that the 
upper bounds \eqref{eq:omgbnd} and \eqref{eq:backErrBnd} of
$\omega_{n_e}$ and $\delta_{n_e}$
are tight within an order of magnitude.

\begin{table}
\caption{
Numerical stability of TRLan with EED for different tolerances $tol$
(Example~\ref{eg:eg1})
} \label{tab:stable}
\centering
\begin{tabular}{| c | c ||  c || c | c || c | c |} \hline
$tol$ & $\mu$ & $\gap_g$ &
$\omega_{n_e}$ & bound \eqref{eq:omgbnd} &
$\|R_{n_e}\|_\F$ & bound \eqref{eq:backErrBnd}  \\ \hline \hline
$10^{-6}$ & 1.00 & $1.00$ &
$2.37\cdot 10^{-6}$ & $1.59\cdot 10^{-5}$ &
$7.87\cdot 10^{-6}$ & $2.24\cdot 10^{-5}$  \\
$10^{-8}$  & 1.00& $1.00$ &
$1.78\cdot 10^{-8}$ & $1.58\cdot 10^{-7}$ &
$7.95\cdot 10^{-8}$ & $2.24\cdot 10^{-7}$  \\
$10^{-10}$  & 1.00& $1.00$ &
$1.82\cdot 10^{-10}$ & $1.58\cdot 10^{-9}$ &
$7.94\cdot 10^{-10}$ & $2.24\cdot 10^{-9}$ \\ \hline
\end{tabular}
\end{table}

\begin{figure}
\centering
\includegraphics[width=0.50\textwidth,trim=0mm 0mm 15mm 11mm,clip=true]
{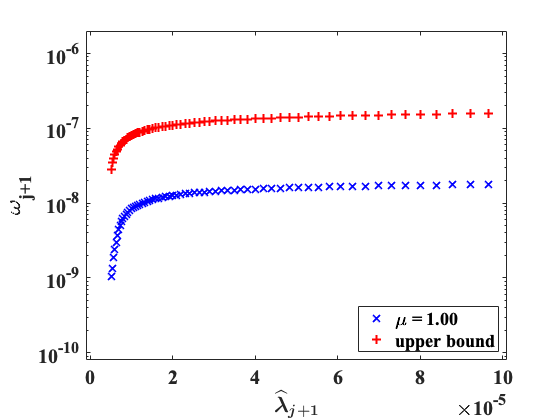}
\caption{The loss of orthogonality $\omega_{j+1}$ 
and the upper bound \eqref{eq:omgbnd} of $\omega_{j+1}$ 
against the 
computed eigenvalues $\fl{\lam}_{j+1}$
for $2\leq j+1 \leq n_e$, $tol = 10^{-8}$
(Example~\ref{eg:eg1}). 
} \label{fig:synthetic_omgs}
\end{figure}


}\end{example} 
 
\begin{example} \label{eg:eg2} 
{\rm  
In this example we illustrate that 
improperly chosen shifts $\sigma_j$ 
may lead to instability of the EED procedure. 

In Remark~\ref{remark:instability}, we stated that
if the shifts $\sigma_j$ are chosen 
such that the spectral gaps $\gap_j$ are too small, 
i.e., $\gap_j \ll {\tt Anorm}$, then 
the loss of orthogonality of the  computed
eigenvectors could be amplified by a factor of $\gap_j^{-1}\cdot{\tt Anorm}$. 
Here is an example using the same diagonal matrix $A$ in Example~\ref{eg:eg1}. 
The combination of TRLan and EED is used
to compute the $n_e=65$ eigenvalues in the interval 
$\mc{I} =  [0,\, 10^{-4}]$.
Let us set the shifts $\sigma_j = \mu - \fl{\lam}_j$ 
with $\mu = 2\cdot10^{-4}$, which is much smaller than the recommended
value of $\mu = \fl{\lam}_{1} + \|A\|_2 \approx 1.00$.
Numerical results are summarized in Table \ref{tab:unstable1}, 
where the tolerance $tol = 10^{-8}$ for TRLan. 
We observe that $\gap_j = O(\gap_g) \ll {\tt Anorm}$, 
and the loss of orthogonality of the  computed eigenvectors is indeed amplified
by a factor of $\gap_g^{-1}\cdot{\tt Anorm}$. 
We note that since $\tau_j = O(1)$, the symmetric backward error norms 
$\delta_{n_e} = O(tol\cdot{\tt Anorm})$. 

\begin{table}
\caption{
Instability of TRLan with EED 
when the spectral gaps $\gap_j = O(\gamma_g)$ are too small.
} \label{tab:unstable1}
\centering
\begin{tabular}{| c | c ||  c || c | c || c | c |} \hline
$tol$ & $\mu$ & $\gap_g$ &
$\omega_{n_e}$ & bound \eqref{eq:omgbnd} &
$\|R_{n_e}\|_\F$ & bound \eqref{eq:backErrBnd} \\ \hline \hline
$10^{-6}$ & $2\cdot10^{-4}$ & $10^{-4}$ &
$8.26\cdot 10^{-3}$ & $1.79\cdot 10^{-1}$ &
$8.00\cdot 10^{-6}$ & $3.29\cdot 10^{-5}$  \\
$10^{-8}$ & $2\cdot10^{-4}$ & $10^{-4}$ &
$8.28\cdot 10^{-5}$ & $1.53\cdot 10^{-3}$ &
$7.96\cdot 10^{-8}$ & $3.22\cdot 10^{-7}$  \\
$10^{-10}$ & $2\cdot10^{-4}$ & $10^{-4}$ &
$8.27\cdot 10^{-7}$ & $1.52\cdot 10^{-5}$ &
$7.95\cdot 10^{-10}$ & $3.22\cdot 10^{-9}$  \\
\hline
\end{tabular}
\end{table}

\medskip
In Remark~\ref{remark:instability}, we also stated that
if the shifts $\sigma_j$ are chosen such that the shift-gap ratios
$\tau_j$ are too large, then the symmetric backward error norm 
$\delta_j$ could be amplified by a factor of $\tau_j$. 
Here is an example. We flip the sign of the diagonal elements 
of $A$ defined in Example~\ref{eg:eg1}, and set $n = 200$. 
We compute $n_e = 74$ eigenvalues in the interval 
$\mc{I} = [-1.0, -0.5001]$
using TRLan with EED procedure.  The computed $2$-norm of $A$ 
is ${\tt Anorm} = 1.00$.

\begin{figure}
\centering
\includegraphics[width=0.48\textwidth,trim=0mm 0mm 0mm 11mm,clip=true]
{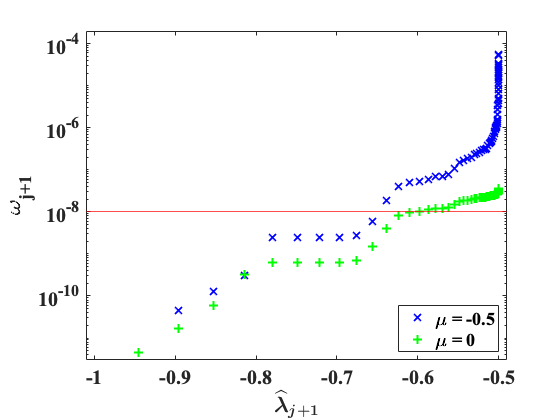}
\includegraphics[width=0.48\textwidth,trim=0mm 0mm 0mm 11mm,clip=true]
{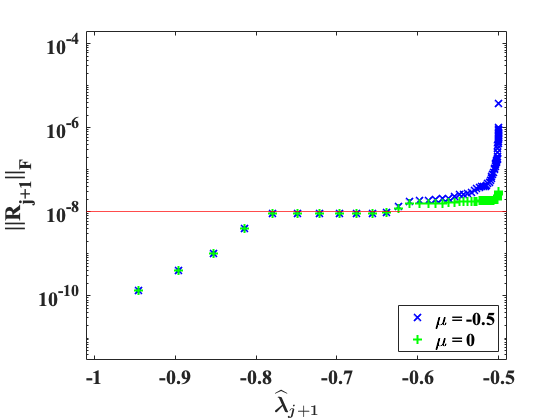}
\caption{
The loss of orthogonality $\omega_{j+1}$ (left) 
and the residual norm $\|R_{j+1}\|_\F$ (right) 
against the computed eigenvalues $\fl{\lam}_{j+1}$
for $2\leq j+1 \leq n_e$. The red lines are 
$tol$ (left) and $tol\cdot {\tt Anorm}$ (right).
}\label{fig:example2}
\end{figure}

Instead of the choice \eqref{eq:shiftchoicefinal}
for the shifts $\sigma_j$, we set $\sigma_j = \mu - \fl{\lam}_j$ 
with $\mu = -0.5$.
The blue $\times$-lines in Figure~\ref{fig:example2} are 
the loss of orthogonality and the residual norms for the computed 
eigenpairs $(\fl{\lam}_{j+1},\fl{v}_{j+1})$ for $2 \leq j+1 \leq n_e$. 
As we observe that for the first $6$ computed eigenvalues in the 
subinterval $[-1.0,-0.75]$ of $\mc{I}$, 
since the spectral gap $\gap_j \geq 0.25$ and 
the shift-gap ratio $\tau_j \leq 2$, 
the computed eigenpairs are backward stable 
with $\omega_{6} = 2.48\cdot 10^{-9} = O(tol)$
and $\|R_{6}\|_\F = 9.05\cdot 10^{-9} = O(tol\cdot {\tt Anorm})$.
However, for the computed eigenvalues in the 
subinterval $[-0.75, -0.5001]$ of $\mc{I}$,  
the computed eigenpairs are not backward stable
due to the facts that the spectral gaps $\gap_j$ 
become small, $\gap_j \approx 1.03\cdot 10^{-4}$,
and the shift-gap ratios $\tau_j$ grows up to 
$\tau_j \approx 4.86\cdot10^3$.
Consequently, the loss of orthogonality $\omega_{n_e}$ 
and the residual norm $\|R_{n_e}\|_\F$ are 
increased by a factor of up to $10^3$, respectively.
The stability are restored if the 
shifts are chosen according to 
the recommendation \eqref{eq:shiftchoicefinal} as shown
by the green $+$-lines in Figure~\ref{fig:example2}.
} 
\end{example}


\begin{example} \label{eg:eg3} 
{\rm  
In this example, we demonstrate the numerical stability of 
TRLan with the EED procedure 
for a set of large sparse symmetric matrices from applications.

The statistics of the matrices are summarized in Table \ref{tab:realproblems}, 
where $n$ is the size of the matrix, 
nnz is the number of nonzero entries of the matrix,
$[\lam_{\min}, \lam_{\max}]$ is the spectrum range,
and $n_e$ is the number of eigenvalues in the interval 
$\mc{I} = [\lam_{\rm low}, \lam_{\rm upper}]$. 
The quantities $n_e$ are calculated 
by computing the inertias of the shifted matrices 
$A - \lam_{\rm upper}I$. 
The {\tt Laplacian} is the negative 2D Laplacian on 
a $200$-by-$200$ grids with Dirichlet boundary 
condition \cite{Knyazev2010laplace}.
The {\tt worms20} is the graph Laplacian worms20\_10NN 
in machine learning datasets \cite{davis2011}.
The 
{\tt SiO, Si34H36, Ge87H76} and {\tt Ge99H100} 
are Hamiltonian matrices from PARSEC collection \cite{davis2011}.

\begin{table}
\caption{Statistics of the test matrices.}  
\label{tab:realproblems}
\centering
\begin{tabular}{| c | c | c | c | c | c |}
\hline
matrix & $n$ & nnz & $[\lam_{\min}, \lam_{\max}]$ & 
$[\lam_{\rm low}, \lam_{\rm upper}]$ 
& $n_e$  \\ \hline\hline  
{\tt Laplacian}  & $40,000$ & $199,200$ & $[0,7.9995]$  &  $[0, 0.07]$  & $205$  \\
{\tt worms20}  & $20,055$ & $260,881$ & $[0,6.0450]$ & $[0,0.05]$ & $289$  \\
{\tt SiO}    & $33,401$ & $1,317,655$ & $[-1.6745,84.3139]$ & $[-1.7, 2.0]$  &  $182$  \\
{\tt Si34H36}   & $97,569$ & $5,156,379$ & $[-1.1586,42.9396]$ & $[-1.2,0.4]$ & $310$ \\
{\tt Ge87H76}  & $112,985$ & $7,892,195$ & $[-1.214,32.764]$ & $[-1.3,-0.0053]$ &  $318$ \\
{\tt Ge99H100} & $112,985$ & $8,451,395$ & $[-1.226,32.703]$ & $[-1.3,-0.0096]$ &  $372$ \\
\hline
\end{tabular}
\end{table}

We run TRLan with a maximal number $m$ of Lanczos vectors
to compute the lowest eigenpairs of the matrix $\fl{A}_j$.
The convergence test is performed at each restart of TRLan.
All the converged eigenvalues in the interval $\mc{I}$
are shifted by EED. Meanwhile, we also keep a maximal number $m_0$ of
the lowest unconverged eigenvectors as the starting vectors of
TRLan for the matrix $\fl{A}_{j+1}$. 
All the eigenvalues in $\mc{I}$ are assumed 
to be computed when the lowest converged eigenvalue is outside 
the interval $\mc{I}$. This combination of TRLan 
is referred to as TRLED.

TRLED is compiled using the \verb|icc| compiler (version 2021.1)
with the optimization flag \verb|-O2|,
and linked to BLAS and LAPACK available in Intel Math Kernel Library
(version 2021.1.1). The experiments are conducted on a MacBook
with $1.6$ GHz Intel Core i5 CPU and 8GB of RAM.

For numerical experiments, we set the maximal number of Lanczos vectors 
$m=150$. When starting TRLED for $\fl{A}_{j+1}$, 
the maximal number of the starting vectors is $m_0=75$.
The convergence tolerance for the residual norm was set to 
$tol=10^{-8}$ as a common practice for solving large scale eigenvalue
problems with double precision; see for example \cite{LiXi2016}.

Numerical results of TRLED are summarized in Table \ref{tab:trleed}, where
the 2nd column is the number $\fl{n}_e$ of 
the computed eigenpairs $(\fl{\lam}_i,\fl{x}_i)$ in the interval $\mc{I}$,
the 3rd column is the number $j_{\max}$ of steps of EED performed,
the 4th column is the loss of orthogonality $\omega_{\fl{n}_e}$,
and the 5th column is the relative residual norm 
$\|R_{\fl{n}_e}\|_\F/{\tt Anorm}$ 
of the computed eigenpairs 
$(\fl{\Lambda}_{n_e},\fl{V}_{n_e})$.
From the quantities $n_e$ in Table~\ref{tab:realproblems} 
and $\fl{n}_e$ in Table~\ref{tab:trleed}, 
we see that for all test matrices 
the eigenvalues in the prescribed intervals $\mc{I}$ 
are successfully computed with the desired backward stability. 

\begin{table}
\caption{Numerical results of TRLED.} \label{tab:trleed}
\centering
\begin{tabular}{|*5{c |}| c | c |}
\hline
\multirow{2}{*}{matrix} & \multirow{2}{*}{$\fl{n}_e$}
& \multirow{2}{*}{$j_{\max}$} 
& \multirow{2}{*}{$\omega_{\fl{n}_e}$} 
& \multirow{2}{*}{$\|R_{\fl{n}_e}\|_{\rm F} / {\tt Anorm}$} 
& \multicolumn{2}{|c|}{CPU time (sec.)} \\ \cline{6-7}
& & & & & TRLED & TRLan \\  \hline\hline
{\tt Laplacian} & $205$ & $60$ & $1.93\cdot 10^{-8}$ & $6.33\cdot 10^{-8}$ & $66.5$  & $86.0$ \\  
{\tt worms20}  & $289$ & $86$ & $2.63\cdot 10^{-8}$ & $7.24\cdot 10^{-8}$ & $57.3$  & $74.8$  \\  
{\tt SiO}           & $182$ & $41$ & $2.33\cdot 10^{-8}$ & $4.71\cdot 10^{-8}$ & $42.4$  & $47.1$  \\  
{\tt Si34H36}   & $310$ & $72$ & $3.41\cdot 10^{-8}$ & $7.50\cdot 10^{-8}$ & $309.9$  & $310.4$  \\  
{\tt Ge87H76}  & $318$ & $66$ & $4.08\cdot 10^{-8}$ & $8.50\cdot 10^{-8}$ & $388.7$  & $421.0$  \\
{\tt Ge99H100} & $372$ & $74$ & $3.65\cdot 10^{-8}$ & $7.63\cdot 10^{-8}$ & $501.1$  & $533.4$  \\
\hline
\end{tabular}
\end{table}

The left plot of Figure~\ref{fig:example3} is a profile of the number 
of converged eigenvalues at each external deflation of a total of 74 
EEDs for the matrix {\tt Ge99H100}.  The right plot of 
Figure~\ref{fig:example3} shows the relative residual norms of all 
372 computed eigenpairs in the interval.  We observe that a large 
number of converged eigenvalues are deflated and shifted away 
at some EED steps.

\begin{figure}
\centering
\includegraphics[width=0.48\textwidth,trim=0mm 0mm 0mm 6mm,clip=true]
{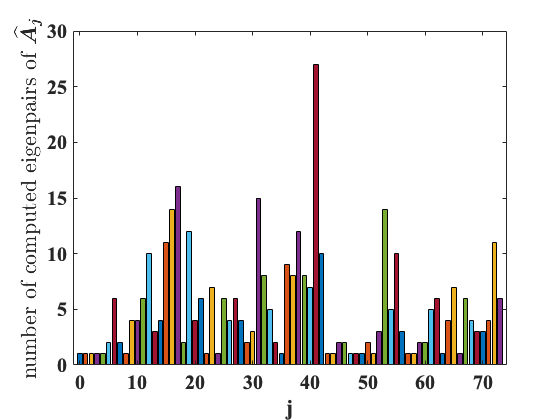} 
\includegraphics[width=0.48\textwidth,trim=0mm 0mm 0mm 6mm,clip=true]
{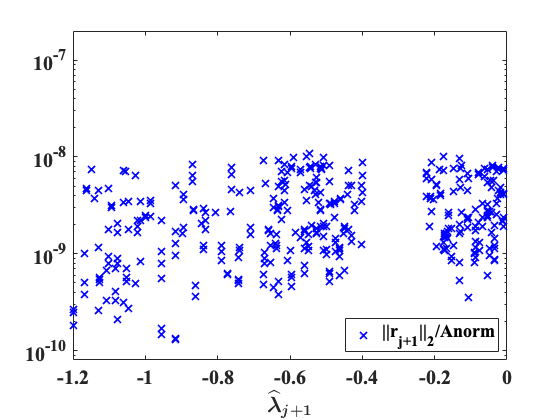} 
\caption{
The number of deflated eigenpairs at each EED for
the matrix {\tt Ge99H100} (left). The relative residual norms of 
372 computed eigenpairs (right).
}\label{fig:example3}
\end{figure}

To examine whether the multiple explicit external deflations lead to 
a significant increase in execution time, in the 6th and 7th columns 
of Table~\ref{tab:trleed}, we record the CPU time of TRLED and TRLan 
for computing all eigenvalues in the same intervals.
For TRLan, we set the maximal number of Lanczos vectors to 
$n_e+150$. The restart scheme with \verb|restart=1| is used. 
TRLan is compiled and executed under the same setting as TRLED.  
%
%
We observe comparable execution time of TRLED and TRLan.
In fact, TRLED is slightly faster. 
We think this might be 
due to the facts that TRLED uses fewer Lanczos vectors,
which compensates the cost of the matrix-vector products 
$y:=\fl{A}_j x$ and the warm start with 
those unconverged eigenvectors from the previous deflated matrix 
$\fl{A}_{j-1}$. These are subjects of future study.
}\end{example}

\section{Concluding remarks} \label{sec:conclude}
Based on the governing equations of the EED procedure in 
finite precision arithmetic, we derived the 
upper bounds on the loss of the orthogonality of the computed eigenvectors 
and the symmetric backward error norm of the computed eigenpairs 
in terms of two key quantities, namely  
the spectral gaps and the shift-gap ratios. 
Consequently, we revealed the required conditions of 
these two key quantities for the backward stability of the EED procedure. 
We present a practical strategy on the  
dynamically selection of the shifts such that 
the spectral gaps and the shift-gap ratios 
satisfy the required conditions. 

There are a number of topics that can be explored further.
One of them is the acceleration effect of the warm starting 
the EIGSOL with the unconverged eigenvectors of $\fl{A}_{j-1}$ for
computing of the lowest eigenpairs of $\fl{A}_j$.
Another one is how to efficiently perform 
the matrix-vector products $ y:=\fl{A}_j x$ in EIGSOL 
by exploiting the sparse plus low rank structure of the
matrix $\fl{A}_j$. Communication-avoid algorithms for 
sparse-plus-low-rank matrix-vector products 
can be found in \cite{Leiserson:1997,Knight:2014}.
A preliminary numerical study for high performance computing 
is reported in \cite{Yamazaki2019}.


%
\bibliographystyle{plain}
\bibliography{EED}

\end{document}